\documentclass[12pt,reqno]{amsart}
\usepackage{amsaddr}

\usepackage{a4wide}
\usepackage{amsthm}
\usepackage{amsmath}
\usepackage{amssymb}
\usepackage[all,cmtip]{xy}
\usepackage{color}
\usepackage[T1]{fontenc}
\usepackage{enumerate}

\usepackage{hyperref}

\usepackage{mathtools}

\theoremstyle{plain}
\newtheorem{thm}{Theorem}
\newtheorem{lem}[thm]{Lemma}
\newtheorem{prop}[thm]{Proposition}
\newtheorem{cor}[thm]{Corollary}

\theoremstyle{definition}
\newtheorem{defn}[thm]{Definition}
\newtheorem{exmp}[thm]{Example}

\newtheorem{rem}[thm]{Remark}

\DeclareMathOperator{\identity}{id}

\DeclareMathOperator{\supp}{Supp}

\newcommand{\Z}{\mathbb{Z}}

\begin{document}

\title[Simple Non-Associative Graded Rings]{Simple Graded Rings, Non-associative Crossed Products and Cayley-Dickson Doublings}

\author{Patrik Nystedt}
\address{Department of Engineering Science,
University West,
SE-46186 Trollh\"{a}ttan, Sweden}
\email{patrik.nystedt@hv.se}

\author{Johan \"{O}inert}
\address{Department of Mathematics and Natural Sciences,
Blekinge Institute of Technology,
SE-37179 Karlskrona, Sweden}
\email{johan.oinert@bth.se}


\begin{abstract}
We show that if a non-associative unital ring is graded by a hypercentral group,
then the ring is simple if and only if 
it is graded simple and the center of the ring is a field. 
Thereby, we extend a result by Jespers 
to a non-associative setting.
By applying this result to non-associative crossed products,
we obtain non-associative analogues of results by
Bell, Jordan and Voskoglou.
We also apply our result to Cayley-Dickson doublings, 
thereby obtaining a new proof of a classical result by McCrimmon.
\end{abstract}

\maketitle

\pagestyle{headings}

\section{Introduction}

Throughout this article, let $R$ be a (not necessarily associative) unital ring
whose multiplicative identity element is denoted by $1$.
Let $G$ be a multiplicatively written group with identity element $e$.
If there are additive subgroups $R_g$, for $g \in G$,
of $R$ such that $R = \oplus_{g \in G} R_g$
and $R_g R_h \subseteq R_{gh}$, for $g,h \in G$, then we shall say that $R$
is \emph{graded by $G$} (or \emph{$G$-graded}).
If in addition $R_g R_h = R_{gh}$, for $g,h \in G$,
then $R$ is said to be \emph{strongly graded by $G$} (or \emph{strongly $G$-graded}).

The investigation of associative graded rings has
been carried out by many authors (see e.g.
\cite{nastasescu1982,nastasescu2004}
and the references therein).
Since many ring constructions are special cases of graded rings, 
e.g. group rings, twisted group rings, (partial) skew group rings and 
crossed product algebras (by twisted partial actions), the theory of graded rings can be applied 
to the study of other less general constructions, 
giving new results for several constructions simultaneously, 
and unifying theorems obtained earlier. 

An important problem in the investigation of graded rings is 
to explore how properties of $R$ are connected to properties of subrings of $R$.
In the associative case, many results of this sort are known for finiteness conditions, nil and radical properties, semisimplicity, semiprimeness and semiprimitivity 
(see e.g. \cite{kelarev1995, kelarev1998}).

In this article, we shall focus on the property of $R$ being {\it simple}.
An obvious {\it necessary} condition for simplicity of $R$
is that the center of $R$, denoted by $Z(R)$, is a field.
This condition is of course not in general a {\it sufficient}
condition for simplicity of $R$.
However, if it is combined with {\it graded simplicity},
then in the associative setting Jespers \cite{jespers1993} has shown (see Theorem \ref{jespers}) that, 
for many groups, it is.
To be more precise, 
recall that an ideal $I$, always assumed to be 
two-sided, of $R$ is called \emph{graded} (or \emph{$G$-graded})
if $I = \oplus_{g \in G} (I \cap R_g)$.
The ring $R$ is called {\it graded simple} (or \emph{$G$-graded simple})
if the only graded ideals of $R$ are
$\{ 0 \}$ and $R$ itself. 
The group $G$ is called \emph{hypercentral}
if every non-trivial factor group of $G$ has a non-trivial center.
Note that all abelian groups and all 
nilpotent groups are hypercentral (see \cite{robinson1972}).

\begin{thm}[Jespers \cite{jespers1993}]\label{jespers}
If an associative unital ring is graded by a hypercentral group,
then the ring is simple if and only if 
it is graded simple and the center of the ring is a field. 
\end{thm}

This result has far-reaching applications.
Indeed, Jespers \cite{jespers1993} has applied 
Theorem \ref{jespers} to obtain necessary and sufficient
conditions for simplicity of
crossed product algebras, previously obtained by 
Bell \cite{bell1987} (see Theorem \ref{bell}), 
and for skew group rings, previously obtained by 
Voskoglou \cite{voskoglou1987} and 
Jordan \cite{jordan1984} (see Theorem \ref{voskoglou}).

\begin{thm}[Bell \cite{bell1987}]\label{bell}
If $T \rtimes_{\sigma}^{\alpha} G$ is a crossed product,
where $G$ is a torsion-free hypercentral group and $T$ is $G$-simple, 
then the following three conditions are equivalent:
(i) $T \rtimes_{\sigma}^{\alpha} G$ is simple;
(ii) $Z( T \rtimes_{\sigma}^{\alpha} G ) = Z(T)^G$; 
(iii) There do not exist $u \in T^{\times}$
and a non-identity $g \in Z(G)$ such that for every
$h \in G$ and every $t \in T$, the relations
$\sigma_h(u) = \alpha(g,h)^{-1} \alpha(h,g) u$ and
$\sigma_g(t) = u t u^{-1}$ hold. 
\end{thm}

\begin{thm}[Jordan \cite{jordan1984}, Voskoglou \cite{voskoglou1987}]\label{voskoglou}
A skew Laurent polynomial ring
\begin{displaymath}
	T[x_1,\ldots,x_n,x_1^{-1},\ldots,x_n^{-1} ; \sigma_1,\ldots,\sigma_n]
\end{displaymath}
is simple if and only if $T$ is $\sigma$-simple
and there do not exist $u \in \cap_{i=1}^n ( T^{\times} )^{\sigma_i}$
and a non-zero $(m_1,\ldots,m_n) \in \mathbb{Z}^n$
such that for every $t \in T$,
the relation $( \sigma_1^{m_1} \circ \cdots \circ \sigma_n^{m_n} )(t) = 
u t u^{-1}$ holds. 
\end{thm}

The main goal of this article is to prove the following non-associative
version of Theorem~\ref{jespers}.

\begin{thm}\label{maintheorem}
If a non-associative unital ring is graded by a hypercentral group,
then the ring is simple if and only if 
it is graded simple and the center of the ring is a field.  
\end{thm}

The secondary goal of this article is to  
use Theorem \ref{maintheorem} to obtain, on the one hand, non-associative
versions of Theorem \ref{bell} and Theorem \ref{voskoglou}
(see Theorem \ref{nonassociativebell} and
Theorem \ref{nonassociativevoskoglou}), and,
on the other hand,
a new proof of a classical result 
by McCrimmon concerning Cayley-Dickson doublings (see Theorem \ref{genmccrimmon});
recall that these algebras are non-associative and, in a natural way, graded.
Here is an outline of this article.

In Section \ref{sectionnonassociativerings},
we gather some well-known facts from non-associative ring
and module theory that we need in the sequel.
In particular, we state our conventions concerning
modules over non-associative rings and
what a basis should mean in that situation.
In Section \ref{freemagma}, 
we recall some notation and
properties
concerning the 
free magma on a set, for use in subsequent sections.
In Section \ref{simplicity}, 
we show that Theorem \ref{jespers}
can be extended to non-associative rings by proving Theorem~\ref{maintheorem}.
In Section \ref{nonassociativecp}, 
we apply Theorem \ref{maintheorem} to non-associative
crossed products and generalize Theorem~\ref{bell} and Theorem~\ref{voskoglou}
(see Theorem~\ref{nonassociativebell} and Theorem~\ref{nonassociativevoskoglou}).
In Section \ref{cayleydickson}, 
we show that a classical result by McCrimmon
(see Theorem \ref{genmccrimmon}) can be 
deduced via Theorem \ref{maintheorem}.

\section{Non-associative rings}\label{sectionnonassociativerings}

In this section, we recall some notions from non-associative
ring theory that we need in subsequent sections. 
Although the results stated in this
section are presumably rather well-known, we have, for the convenience of the reader,
nevertheless chosen to include proofs of these statements.

Throughout this section, 
$R$ denotes a non-associative ring.
By this we mean that $R$ is an additive abelian group on which a multiplication
is defined, satisfying left and right distributivity.
We always assume that $R$ is unital and that the multiplicative identity element
of $R$ is denoted by $1$.
The term ''non-associative'' should be interpreted 
as ''not necessarily associative''.
Therefore all associative rings are non-associative.
If a ring is not associative,
we will use the term ''not associative ring''. 

By a {\it left module} over $R$ we mean an additive group $M$
equipped with a bi-additive map 
$R \times M \ni (r,m) \mapsto rm \in M$.
A subset $B$ of $M$ is said to be a \emph{basis for $M$},
if for every $m \in M$, there are unique $r_b \in R$, for $b \in B$,
such that $r_b = 0$ for all but finitely many $b \in B$,
and $m = \sum_{b \in B} r_b b$.
{\it Right modules} over $R$ and their bases are defined in an analogous manner.

Recall that the \emph{commutator} $[\cdot,\cdot] : R \times R \rightarrow R$ 
and the \emph{associator} $(\cdot,\cdot,\cdot) : R \times R \times R \rightarrow R$ 
are defined by $[r,s]=rs-sr$ and
$(r,s,t) = (rs)t - r(st)$ for all $r,s,t \in R$, respectively.
The \emph{commuter} of $R$, denoted by $C(R)$,
is the subset of $R$ consisting
of elements $r \in R$ such that $[r,s]=0$
for all $s \in R$.
The \emph{left}, \emph{middle} and \emph{right nucleus} of $R$,
denoted by $N_l(R)$, $N_m(R)$ and $N_r(R)$, respectively, are defined by 
$N_l(R) = \{ r \in R \mid (r,s,t) = 0, \ \mbox{for} \ s,t \in R\}$,
$N_m(R) = \{ s \in R \mid (r,s,t) = 0, \ \mbox{for} \ r,t \in R\}$, and
$N_r(R) = \{ t \in R \mid (r,s,t) = 0, \ \mbox{for} \ r,s \in R\}$.
The \emph{nucleus} of $R$, denoted by $N(R)$,
is defined to be equal to $N_l(R) \cap N_m(R) \cap N_r(R)$.
From the so-called \emph{associator identity} 
$u(r,s,t) + (u,r,s)t + (u,rs,t) = (ur,s,t) + (u,r,st)$,
which holds for all $u,r,s,t \in R$, it follows that
all of the subsets $N_l(R)$, $N_m(R)$, $N_r(R)$ and $N(R)$
are associative subrings of $R$.
The \emph{center} of $R$, denoted by $Z(R)$, is defined to be equal to the 
intersection $N(R) \cap C(R)$.
It follows immediately that $Z(R)$ is an associative, unital
and commutative subring of $R$.

\begin{prop}\label{intersection}
The following three equalities hold:
\begin{displaymath}
	Z(R) = C(R) \cap N_l(R) \cap N_m(R) = C(R) \cap N_l(R) \cap N_r(R) = C(R) \cap N_m(R) \cap N_r(R).	
\end{displaymath}
\end{prop}

\begin{proof}
We only show the first equality.
The other two equalities are shown 
in a similar way and are therefore left to the reader.
It is clear that $Z(R) \subseteq C(R) \cap N_l(R) \cap N_m(R)$.
Now we show the reversed inclusion. 
Take $r \in C(R) \cap N_l(R) \cap N_m(R)$. 
We need to show that $r \in N_r(R)$.
Take $s,t \in R$.
We wish to show that $(s,t,r)=0$, i.e. $(st)r = s(tr)$.
Using that $r\in C(R) \cap N_l(R) \cap N_m(R)$
we get $(st)r = r(st) = (rs)t = (sr)t = s(rt) = s(tr)$.
\end{proof}

\begin{prop}\label{centerInvClosed}
If $r \in Z(R)$ and $s \in R$ satisfy $rs = 1$,
then $s \in Z(R)$.
\end{prop}

\begin{proof}
Let $r\in Z(R)$ and suppose that $rs=1$.
First we show that $s \in C(R)$.
To this end, take $u \in R$.
Then $su = (su)1 = (su)(rs) = ((su)r)s = (r(su))s=
((rs)u)s = (1u)s = us$ and hence $s \in C(R)$.
By Proposition \ref{intersection}, we are done if we can show
$s \in N_l(R) \cap N_m(R)$. To this end, take $v \in R$.
Then $s(uv) = s((1u)v)= s(((rs) u)v) = 
(rs) ( (su) v ) = 1( (su)v ) = (su)v$
which shows that $s \in N_l(R)$. 
We also see that
$(us)v = (us)(1v) = (us) ( (rs) v ) =
( u (rs) ) (sv) = (u1)(sv) = u(sv)$
which shows that $s \in N_m(R)$.
\end{proof}

\begin{prop}\label{centerfield}
If $R$ is simple, then $Z(R)$ is a field.
\end{prop}

\begin{proof}
This follows from Proposition \ref{centerInvClosed}.
\end{proof}

\begin{rem}
Denote by $R^{\times}$ the set of all elements of $R$ which
have two-sided multiplicative inverses.
Notice that $R^\times$ is not necessarily a group.
\end{rem}

\begin{prop}\label{propNRunit}
The equality $N(R) \cap R^{\times} = N(R)^{\times}$ holds.
\end{prop}

\begin{proof}
The inclusion $N(R) \cap R^{\times} \supseteq N(R)^{\times}$ is clear.
Now we show the reversed inclusion.
Take $r,s \in R$ satisfying $r \in N(R)$ and $rs=sr=1$.
We need to show that $s \in N(R)$.
Take $u,v \in R$.
First we show that $s \in N_l(R)$.
\begin{align*}
s(uv) = 
s( (1u) v ) = 
s( ((rs)u) v ) =
s( ( r (su) ) v ) =
s( r ( (su) v ) ) 
= (sr) ( (su) v ) = (su) v.
\end{align*}
Now we show that $s \in N_m(R)$.
\begin{displaymath}
	(us)v = 
(us) ( (rs) v ) =
(us) ( r (sv) ) =
( (us)r ) (sv) =
( u (sr) ) (sv) =
u(sv).
\end{displaymath}
Finally, we show that $s \in N_r(R)$.
\begin{displaymath}
	(uv)s = 
( u (v (sr) ) ) s = 
( u ((v s)r) ) s =
((u (vs))r) s =
(u (vs)) (rs) =
u(vs).
\end{displaymath}
\end{proof}

\section{The free magma on a set}\label{freemagma}

Throughout this section, $X$ denotes
a non-empty set
and 
$R$ denotes a non-associative ring.
Recall that we can define the {\it free magma on $X$},
denoted by $M(X)$, as the set of non-associative words in the alphabet $X$ 
with multiplication given by restoring the outer set of parentheses in each multiplicand and 
juxtaposing
(see e.g. \cite[Definition 21B.1]{rowen2008}).
We call the elements of $X$ \emph{letters} and they are \emph{words} of length 1.
Inductively, if $\alpha$ and $\beta$ are words of length $m$ and $n$
respectively, then $(\alpha)(\beta)$ is a word of length $m+n$.
Furthermore, two words $(\alpha)(\beta)$ and $(\alpha')(\beta')$ 
are considered to be equal precisely when $\alpha=\alpha'$ and $\beta=\beta'$.
Note that in each word $\alpha$ there are only a finite number of letters
$x_1,\ldots,x_n$, where $x_1$ is the first letter (to the left) and $x_n$ is the last letter (to the right). We then say that $\alpha$ is of \emph{length $n$}
which we indicate by writing $\alpha = \alpha(x_1,\ldots,x_n)$. 
For an arbitrary word $\alpha$ of length $n$ we define a map
\begin{displaymath}
	\overline{\alpha} : R^n \to R
\end{displaymath}
in the following way.
Given $(r_1,\ldots,r_n) \in R^n$,
insert $r_1$ in the first position (to the left)
in the expression of $\alpha$
and then successively continue from the left to the right by inserting $r_j$ in the $j$th position,
until $r_n$ has been inserted in $n$th position (to the far right).
Finally, 
evaluate the arisen expression in $R$ and define $\overline{\alpha}(r_1,\ldots,r_n)$ to be equal to it.
The element $\overline{\alpha}(r_1,\ldots,r_n)$ will be referred to as a \emph{specialization of $\alpha$ at $(r_1,\ldots,r_n)$} or simply a \emph{specialization}.

\begin{rem}\label{rem:def+multilinear}
(a)
We want to emphasize that, in the above procedure describing the definition of $\overline{\alpha}$, we are not primarily substituting letters in $\alpha$ by elements of $R$, but rather inserting elements of $R$ in \emph{positions} in the word $\alpha$.
A specific letter may occur more than once in a single word
and thus the above approach is necessary
in order to make $\overline{\alpha}$ defined on all of $R^n$.

(b)
For any word $\alpha \in M(X)$ of length $n$, the map $\overline{\alpha} : R^n \to R$ is
\emph{$n$-additive},
i.e. additive in each of its arguments separately.

(c)
If $\alpha$ and $\beta$ are words of length $n$ and $m$, respectively, then
for any $(r_1,\ldots,r_n) \in R^n$ and $(s_1,\ldots,s_m) \in R^m$,
\begin{displaymath}
	\overline{\alpha}(r_1,\ldots,r_n) \cdot
	\overline{\beta}(s_1,\ldots,s_m)
	=
	\overline{(\alpha)(\beta)}(r_1,\ldots,r_n,s_1,\ldots,s_m)
\end{displaymath}
holds in $R$.
\end{rem}

\begin{exmp}
Take distinct elements $a,b,c,d \in X$
and define the words 
\begin{displaymath}
\alpha(a,b,c,d)=a((bc)d),
\quad
\beta(a,b,c,d)=(a(bc))d, \\
 \quad \text{and}
\quad
\gamma(a,b,a,b)=a((ba)b).
\end{displaymath}
Notice that $\alpha$, $\beta$ and $\gamma$ are distinct words
in $M(X)$.
Nevertheless, for any $(r_1,r_2,r_3,r_4) \in R^4$ we get that
$\overline{\alpha}(r_1,r_2,r_3,r_4) = r_1((r_2 r_3) r_4) = \overline{\gamma}(r_1,r_2,r_3,r_4)$.
This shows that, although $\alpha \neq \gamma$ in $M(X)$,
we have $\overline{\alpha} = \overline{\gamma}$ for the corresponding maps.
\end{exmp}

\begin{prop}\label{idealprop}
If $A$ is a subset of $R$, then the ideal of $R$ generated by $A$, 
denoted by $\langle A \rangle$, equals the additive group generated by specializations 
$\overline{\alpha}(r_1,\ldots,r_n)$, where $r_i \in R$, for $i\in \{1,\ldots,n\}$, 
of
words $\alpha(x_1,\ldots,x_n)$,
such that at least one of the $r_i$'s belongs to $A$.
\end{prop}

\begin{proof}
This is self-evident.
\end{proof}

\section{Simplicity}\label{simplicity}

At the end of this section, we show Theorem \ref{maintheorem}.
We adapt the approach taken by Jespers \cite{jespers1993} 
to the non-associative setting.
Throughout this section $R$ denotes a non-associative 
unital ring which is graded by a group $G$. 
Furthermore, $X$ denotes a non-empty set. 
Recall that the elements of $h(R) := \cup_{g \in G} R_g$ are called {\it homogeneous}. 
A specialization $\overline{\alpha}(r_1,\ldots,r_n)$
of a word $\alpha(x_1,\ldots,x_n)$ is said to be
\emph{homogeneous} if all the $r_i$'s are homogeneous.

\begin{rem}
If $R$ is a $G$-graded ring and $N$ is a normal subgroup of $G$, then $R$ carries a natural $G/N$-gradation (cf. \cite[p. 3]{nastasescu2004}).
In particular, any $G$-graded ring has a natural $G/Z(G)$-gradation.
When no confusion can arise about which gradation to consider
it is customary to just speak of e.g. \emph{homogeneous} or \emph{graded simple}.
In this section, however, for clarity we will sometimes explicitly write
e.g. \emph{$G$-graded simple} and \emph{$G/Z(G)$-graded simple}.
\end{rem}

\begin{prop}\label{gradedsum}
Every specialization $\overline{\alpha}(r_1,\ldots,r_n)$
of a word $\alpha(x_1,\ldots,x_n)$
may be written as a sum of homogeneous specializations of $\alpha$.
\end{prop}

\begin{proof}
This follows immediately from Remark~\ref{rem:def+multilinear}(b).
\end{proof}

\begin{prop}\label{gradedidealprop}
If $A$ is a subset of $h(R)$, then $\langle A \rangle$
is a graded ideal which equals 
the additive group generated by homogeneous specializations 
$\overline{\alpha}(r_1,\ldots,r_n)$ of
words $\alpha(x_1,\ldots,x_n)$
such that at least one of the $r_i$'s belongs to $A$.
\end{prop}

\begin{proof}
This follows from Proposition~\ref{idealprop} and
Proposition \ref{gradedsum}.
\end{proof}

If $r \in h(R)$ is homogeneous of degree $g$, then we indicate this by writing $\deg(r)=g$.
The following is a non-associative analogue of \cite[Proposition 1.1.1]{nastasescu2004}.

\begin{prop}\label{NFVOPropNonAssoc}
The following two assertions hold:
\begin{enumerate}[{\rm (a)}]
	\item\label{unit}
$1 \in R_e$.
In particular, $R_e \neq \{0\}$;
	\item\label{inverse}
If $r \in R^\times \cap N(R)$ is homogeneous,
then $r^{-1}$ is homogeneous of degree $\deg(r)^{-1}$.
\end{enumerate}
\end{prop}

\begin{proof}
(a)
Let $1 = \sum_{g \in G} 1_g$.
For $h \in G$ and $r_h \in R_h$
we get that $r_h = r_h \cdot 1 =
\sum_{g \in G} r_h 1_g$.
Thus,
$r_h = r_h 1_e$ and $r_h 1_g = 0$ for every $g \in G\setminus \{e\}$.
In particular, for any $g \in G\setminus \{e\}$, we get that
$1_g = 1 \cdot 1_g = \sum_{h \in G} 1_h 1_g = 0$.
This shows that $1 = 1_e \in R_e$.

(b)
From Proposition \ref{propNRunit} it follows that $R^\times \cap N(R)$
is a group, making $s := r^{-1}$ well-defined.
Suppose that $r \in R_g$, for some $g \in G$.
From \eqref{unit} we get that $1 \in R_e$.
Hence $1 = 1_e = (rs)_e = r s_{g^{-1}}$ and
$1 = 1_e =(sr)_e = s_{g^{-1}} r$.
This shows that $r^{-1}=s_{g^{-1}} \in R_{g^{-1}}$.
\end{proof}

\begin{prop}\label{nonzeroword}
If $R$ is graded simple 
and for some $g \in G$, the additive group $R_g$ is non-zero,
then for any non-zero $r \in h(R)$,
there is a homogeneous specialization $\overline{\alpha}(r_1,\ldots,r_n)$ of a
word $\alpha(x_1,\ldots,x_n)$ such that
$\overline{\alpha}(r_1,\ldots,r_n)$ is a non-zero element in $R_g$
and at least one of the $r_i$'s equals $r$.
\end{prop}

\begin{proof}
Put $A = \{ r \}$.
By Proposition \ref{gradedidealprop}, we get that $\langle A \rangle$
is a non-zero graded ideal of $R$.
Since $R$ is graded simple, we therefore get that $\langle A \rangle = R$.
In particular, this implies that $R_g = \langle A \rangle_g$.
By Proposition \ref{gradedidealprop} again and the fact that $R_g$ is non-zero
the claim follows.
\end{proof}

\begin{defn}
Suppose that $I$ is an ideal of the $G$-graded ring $R$.
Following Jespers \cite{jespers1993}, we say that a subset $M$ of $G$
satisfies the {\it minimal support condition} (MS), with respect to $I$,
if there exists a non-zero $r \in I$ with $\supp(r)=M$,
and there is no proper subset $N$ of $M$
such that $N = \supp(s)$ for some non-zero $s \in I$.
\end{defn}

\begin{prop}\label{product}
Suppose that, for some $h\in G$, $\overline{\alpha}(r_1,\ldots,r_n) \in R_h$ is a non-zero homogeneous
specialization of a
word $\alpha(x_1,\ldots,x_n)$,
where for each $i\in \{1,\ldots,n\}$, $r_i \in R_{h_i}$, for some $h_i \in G$.
If $z_i \in Z(G)$ and $s_i \in R_{h_i z_i}$, for $i \in \{1,\ldots,n\}$, then 
$\overline{\alpha}(s_1,\ldots,s_n) \in R_{h z_1 \cdots z_n}$.
\end{prop}

\begin{proof}
This is shown by induction over the length of $\alpha$.
\end{proof}

\begin{defn}
Let $M$ be a subset of $G$
and let $I$ be a subset of $R$.
We say that $M$ is an {\it $I$-set}
if there is $r \in I$ with $\supp(r) = M$.
We say that $M$ is {\it $Z(G)$-homogeneous}
if $M$ is a subset of a coset of $Z(G)$ in $G$.
\end{defn}

\begin{prop}\label{translation}
Suppose that $R$ is $G$-graded simple, $I$ is
a non-zero $G/Z(G)$-graded ideal of $R$,
$M$ is a non-empty $Z(G)$-homogeneous subset of $G$ and that $g \in G$
satisfies $R_{gM} \neq \{ 0 \}$.
The following two assertions hold:
\begin{enumerate}[{\rm (a)}]
	\item If $M$ is an $I$-set, then there is a non-empty 
$I$-set $N$ such that $N \subseteq gM$;
	\item If $M$ satisfies MS, then $gM$ satisfies MS. 
\end{enumerate} 
\end{prop}

\begin{proof}
Since $R_{gM} \neq \{ 0 \}$, there is $h \in M$
with $R_{gh} \neq \{ 0 \}$.
Notice that an $I$-set is necessarily finite.
Put $m = |M|$.
Since $M$ is $Z(G)$-homogeneous it follows that $M = \{ h z_1 , \ldots ,h z_m \}$
for distinct $z_1,\ldots,z_m \in Z(G)$ with $z_1 = e$.

(a) Suppose that $M$ is an $I$-set.
Then there is a non-zero $r \in I$
with $\supp(r) = M$.
By Proposition \ref{nonzeroword},
there is a non-zero homogeneous specialization
$\overline{\alpha}(r_1,\ldots,r_n)$ of a 
word 
$\alpha(x_1,\ldots,x_n)$ such that 
$\overline{\alpha}(r_1,\ldots,r_n) \in R_{gh}$
and $r_i = r_h$ for some $i \in \{ 1,\ldots,n \}$.
By Proposition \ref{product}, we get that 
\begin{align*}
	I \setminus \{ 0 \} \ni s :=& \overline{\alpha}(r_1,\ldots,r_{i-1},r,r_{i+1},\ldots,r_n) \\
	=& \sum_{j=1}^m \overline{\alpha}(r_1,\ldots,r_{i-1},r_{h z_j},r_{i+1} , \ldots , r_n) \in 
\oplus_{j=1}^m R_{g h z_j} = R_{gM}
\end{align*}
In particular, if we put $N=\supp(s)$, we get that $N$
is a non-empty $I$-set with $N \subseteq gM$.

(b) Suppose that $M$ satisfies MS.
Using (a) we get that there is a non-empty $I$-set $N$ with $N \subseteq gM$.
Suppose that there is a non-empty $I$-set $P$ with $P \subseteq N$.
Since $M$ is an $I$-set and $g^{-1}P \subseteq g^{-1}N \subseteq g^{-1}gM = M$, 
we conclude that $R_{g^{-1}P} \neq \{ 0 \}$.
Therefore, by (a), there is a non-empty $I$-set $Q$
with $Q \subseteq g^{-1}P$.
But then $Q \subseteq M$, which by MS of $M$ implies that $Q = M$.
Therefore $M \subseteq g^{-1}P \subseteq M$ and hence 
$P = N = gM$. This proves that $gM$ satisfies MS
with respect to $I$.
\end{proof}

\begin{prop}\label{change}
Suppose that $R$ is $G$-graded simple, 
$I$ is a non-zero $G/Z(G)$-graded ideal of $R$ and
that $r \in R_{Z(G)}$ has the property that $M=\supp(r)$ satisfies MS with respect to $I$.
If $\overline{\alpha}(s_1,\ldots,s_n)$ is a non-zero homogeneous specialization of a 
word 
$\alpha(x_1,\ldots,x_n)$, with $n \geq 2$, then,
for any 
$g,h \in \supp(r)$, the following equality holds:
\begin{align*}
	\overline{\alpha}(s_1,\ldots,s_{i-1},r_g,s_{i+1},\ldots,s_{j-1},r_h,s_{j+1},\ldots,s_n) = \\
= \overline{\alpha}(s_1,\ldots,s_{i-1},r_h,s_{i+1},\ldots,s_{j-1},r_g,s_{j+1},\ldots,s_n)
\end{align*}
\end{prop}

\begin{proof}
Clearly, $s_k \neq 0$, for $k \in \{1,\ldots,n\}$.
Put
\begin{displaymath}
	t = \overline{\alpha}(s_1,\ldots,s_{i-1},r_g,s_{i+1},\ldots,s_{j-1},r_h,s_{j+1},\ldots,s_n)
\end{displaymath}
and
\begin{displaymath}
	t' =  \overline{\alpha}(s_1,\ldots,s_{i-1},r_h,s_{i+1},\ldots,s_{j-1},r_g,s_{j+1},\ldots,s_n).
\end{displaymath}
Seeking a contradiction, suppose that $t \neq t'$.
Then $t-t' \in I \setminus \{ 0 \}$ and 
$\supp(t-t') \subsetneq h \prod_{1 \leq k \leq n, \ k \neq i,j} \deg(s_k) M$. 
But from Proposition \ref{translation}(b), we get that
$h \prod_{1 \leq k \leq n, \ k \neq i,j} \deg(s_k) M$ satisfies MS
which is a contradiction. Thus, $t = t'$.
Comparing terms of degree
\begin{displaymath}
	gh \prod_{1 \leq k \leq n, \ k \neq i,j} \deg(s_k)	
\end{displaymath}
in the last equality yields the desired conclusion.
\end{proof}

\begin{defn}\label{defc}
Suppose that $R$ is $G$-graded simple and $I$ is a non-zero $G/Z(G)$-graded ideal of $R$.
Furthermore, suppose that there exists
a non-zero
$a \in I \cap R_{Z(G)}$ such that ${\rm Supp}(a)$ satisfies MS, with respect to the $G$-gradation and $I$. 
Fix $z \in {\rm Supp}(a)$.
By the $G$-graded simplicity and Proposition~\ref{gradedidealprop}, there exist $m \in \mathbb{N}$, 
maps $n : \{ 1,\ldots,m \} \to \mathbb{N}$ and $p : \{ 1,\ldots,m \} \to \mathbb{N}$ such that
$p(j) \leq n(j)$, and there exist
words $\alpha_j(x_1, \ldots , x_{n(j)})$, 
for $j \in \{1,\ldots,m\}$, such that
$1 = \sum_{j=1}^m \overline{\alpha_j}( r_{(1,j)} , \ldots , r_{(n(j),j)} )$,
for some $G$-homogeneous $r_{(u,v)} \in h(R)$ satisfying $r_{(p(j),j)} = a_z$, for $j \in \{1,\ldots,m\}$.
From Proposition \ref{NFVOPropNonAssoc}(a) it follows that we may assume that
for all $j \in \{1,\ldots,m\}$, 
\begin{equation}\label{equality}
\prod_{i=1}^{n(j)} {\rm deg}( r_{(i,j)} ) = e.
\end{equation}
For all $g \in {\rm Supp}(a)$, put
$c(g) = \sum_{j=1}^m \overline{\alpha_j}( s_{(1,j)}^g , \ldots , s_{(n(j),j)}^g )$,
where $s_{(i,j)}^g = r_{(i,j)}$, if $i \neq p(j)$, and $s_{(p(j),j)}^g = a_g$,
for all $j \in \{1,\ldots,m\}$ and all $i \in \{1,\ldots,n(j)\}$.
Put $c = \sum_{g \in {\rm Supp}(a)} c(g).$
\end{defn}

\begin{lem}\label{lemmacentral}
With the above notation, we get that
$c \in I \cap R_{Z(G)} \cap Z(R)$ and $a = a_z c$.
\end{lem}

\begin{proof}
Take $g \in {\rm Supp}(a)$. From \eqref{equality} it follows that
${\rm deg}(c(g)) = g z^{-1} \in Z(G)$. Thus $c \in R_{Z(G)}$.
Now we show that $c \in I$.
Using that each $\overline{\alpha_j}$ is $n(j)$-additive it follows that
$c = \sum_{j=1}^m \overline{\alpha_j}( s_{(1,j)} , \ldots , s_{(n(j),j)} )$
where $s_{(i,j)} = r_{(i,j)}$, if $i \neq p(j)$, and $s_{(p(j),j)} = a$,
for all $j \in \{1,\ldots,m\}$ and all $i \in \{1,\ldots,n(j)\}$.
Since $a \in I$ we get that $\overline{\alpha_j}( s_{(1,j)} , \ldots , s_{(n(j),j)} ) \in I$,
for $j \in \{1,\ldots,m\}$. Thus $c \in I$.
Now we show that $c \in Z(R)$. Take $r \in R$.
By Proposition~\ref{gradedidealprop}
and the $G$-graded simplicity of $R$,
there exist $m' \in \mathbb{N}$, 
maps $n' : \{ 1,\ldots,m' \} \to \mathbb{N}$ and $p' : \{ 1,\ldots,m' \} \to \mathbb{N}$ such that
$p'(k) \leq n'(k)$, and there exist
words $\alpha_k'(x_1, \ldots , x_{n'(k)})$, 
for $k \in \{1,\ldots,m'\}$, such that
\begin{displaymath}
	r = \sum_{k=1}^{m'} \overline{\alpha_k'}( r_{(1,k)}' , \ldots , r_{(n'(k),k)}' )
\end{displaymath}
for some $G$-homogeneous $r_{(u,v)}' \in h(R)$ satisfying $r_{(p'(k),k)}' = a_z$, for $k \in \{1,\ldots,m'\}$.
Define $G$-homo\-geneous $s_{(u,v)}'^g \in h(R)$ by saying that
$s_{(i,j)}'^g = r_{(i,j)}'$, if $i \neq p'(k)$, and $s_{(p(k),k)}'^g = a_g$,
for all $k \in \{1,\ldots,m'\}$ and all $i \in \{1,\ldots,n'(k)\}$.
From Proposition \ref{change} we now get that
\begin{align*}
r c &= \sum_{k=1}^{m'} \sum_{g \in {\rm Supp}(a)}  \sum_{j=1}^m   
\overline{\alpha_k'}( r_{(1,k)}' , \ldots , r_{(n'(k),k)}' )  \overline{\alpha_j}( s_{(1,j)}^g , \ldots , s_{(n(j),j)}^g ) \\
&= \sum_{k=1}^{m'} \sum_{g \in {\rm Supp}(a)}  \sum_{j=1}^m
\overline{\alpha_k'}( s_{(1,k)}'^g , \ldots , s_{(n'(k),k)}'^g )  \overline{\alpha_j}( r_{(1,j)} , \ldots , r_{(n(j),j)} ) \\
&= \sum_{k=1}^{m'} \sum_{g \in {\rm Supp}(a)}  
\overline{\alpha_k'}( s_{(1,k)}'^g , \ldots , s_{(n'(k),k)}'^g ) \sum_{j=1}^m \overline{\alpha_j}( r_{(1,j)} , \ldots , r_{(n(j),j)} ) \\
&= \sum_{k=1}^{m'} \sum_{g \in {\rm Supp}(a)}  
\overline{\alpha_k'}( s_{(1,k)}'^g , \ldots , s_{(n'(k),k)}'^g ) \cdot 1  =
\sum_{k=1}^{m'} \sum_{g \in {\rm Supp}(a)}  1 \cdot \overline{\alpha_k'}( s_{(1,k)}'^g , \ldots , s_{(n'(k),k)}'^g ) \\
&= \sum_{k=1}^{m'} \sum_{g \in {\rm Supp}(a)} 
\sum_{j=1}^m \overline{\alpha_j}( r_{(1,j)} , \ldots , r_{(n(j),j)} )  \overline{\alpha_k'}( s_{(1,k)}'^g , \ldots , s_{(n'(k),k)}'^g )
\\
&=
\sum_{k=1}^{m'} \sum_{g \in {\rm Supp}(a)}  \sum_{j=1}^m
\overline{\alpha_j}( s_{(1,j)} , \ldots , s_{(n(j),j)} )   \overline{\alpha_k'}( r_{(1,k)}'^g , \ldots , r_{(n'(k),k)}'^g ) =cr.
\end{align*}
Now we show that $a = a_z c$.
From Proposition \ref{change} again, it follows that
\begin{align*}
a &= a \cdot 1 = 
\sum_{g \in {\rm Supp}(a)} \sum_{j=1}^m a_g \overline{\alpha_j}( r_{(1,j)} , \ldots , r_{(n(j),j)} ) =
\sum_{g \in {\rm Supp}(a)} \sum_{j=1}^m a_z \overline{\alpha_j}( s_{(1,j)}^g , \ldots , s_{(n(j),j)}^g ) \\
&= a_z \sum_{g \in {\rm Supp}(a)} \sum_{j=1}^m \overline{\alpha_j}( s_{(1,j)}^g , \ldots , s_{(n(j),j)}^g ) =
a_z \sum_{g \in {\rm Supp}(a)} c(g) = a_z c.
\end{align*}
\end{proof}

\begin{prop}\label{generated}
If $R$ is $G$-graded simple and $I$
is a $G/Z(G)$-graded ideal of $R$, then
$I = R(I \cap R_{Z(G)} \cap Z(R))$.
\end{prop}

\begin{proof}
The inclusion $I \supseteq R ( I \cap R_{Z(G)} \cap Z(R) )$
is clear. Now we show the inclusion
$I \subseteq R (I \cap R_{Z(G)} \cap Z(R))$.
Take a non-zero $a \in I$.

Case 1: $\supp(a)$ satisfies MS and 
$a \in R_{Z(G)}$. Take $z \in \supp(a)$.
From Lemma \ref{lemmacentral} it follows that
$a = a_z c$ for some $c \in I \cap R_{Z(G)} \cap Z(R)$.

Case 2: $\supp(a)$ satisfies MS.
Since $I$ is $G/Z(G)$-graded, we can assume that $a$ is $Z(G)$-homogeneous.
By Proposition~\ref{NFVOPropNonAssoc}\eqref{unit} and Proposition~\ref{translation},
we thus get that there is a non-zero $b \in I \cap R_{Z(G)}$ and 
$h \in \supp(a)$ such that $\supp(b) = h^{-1} \supp(a)$ satisfies MS.
By case 1, (with $z=e$) we get that 
there is $c \in I \cap R_{Z(G)} \cap Z(R)$
such that $b = b_e c$, $c_e = 1$ and $\supp(c) = \supp(b)$.
Then $a - a_h c \in I$ and $\supp(a - a_h c) \subsetneq \supp(a)$.
Hence, using that $\supp(a)$ satisfies MS, we get that 
$a = a_h c \in R (I \cap R_{Z(G)} \cap Z(R))$.

Case 3: $\supp(a)$ does not satisfy MS.
We prove, by induction, that $a \in R(I \cap R_{Z(G)} \cap Z(R))$.
Suppose that for every non-zero $b \in I$
with $\supp(b) \subsetneq \supp(a)$,
the relation $b \in R(I \cap R_{Z(G)} \cap Z(R))$ holds. 
Take a proper non-empty subset $X$ of $\supp(a)$
and $b \in I$ with $\supp(b)=X$ satisfying MS.
Take $h \in \supp(b)$.
By Case 2, there is $c \in I \cap R_{Z(G)} \cap Z(R)$
with $c_e = 1$ and $\supp(c) = h^{-1} \supp(b)$.
Then $\supp(a - a_h c) \subsetneq \supp(a)$
and $a - a_h c \in I$. Hence, by the induction hypothesis,
we get that $a - a_h c \in R(I \cap R_{Z(G)} \cap Z(R))$.
But since obviously $a_h c \in R(I \cap R_{Z(G)} \cap Z(R))$,
we finally get that $a \in R(I \cap R_{Z(G)} \cap Z(R))$.
\end{proof}

\begin{cor}\label{corgenerated}
If $R$ is $G$-graded simple and $Z(R)$ is a field,
then $R$ is $G/Z(G)$-graded simple.
\end{cor}

\begin{proof}
Suppose that $I$ is a non-zero $G/Z(G)$-graded ideal of $R$.
By Proposition \ref{generated}, $I$ contains 
a non-zero element from $Z(R)$. Since $Z(R)$
is a field this implies that $I = R$.
Thus $R$ is $G/Z(G)$-graded simple.
\end{proof}

\begin{rem}\label{ascending}
Recall that if $G$ is a group with identity
element $e$, then the ascending central series 
of $G$ is the sequence of subgroups $Z_i(G)$, for
non-negative integers $i$, defined recursively by
$Z_0(G) = \{ e \}$ and, given $Z_i(G)$, for
a non-negative integer $i$, 
$Z_{i+1}(G)$ is defined to be the set of 
$g \in G$ such that for every $h \in G$,
the commutator $[g,h]=g h g^{-1} h^{-1}$
belongs to $Z_i(G)$.
For infinite groups this process can be
continued to infinite ordinal numbers
by transfinite recursion. For a limit ordinal
$O$, we define $Z_O(G) = \cup_{i < O} Z_i(G)$.
If $G$ is hypercentral, then
$Z_O(G) = G$ for some limit ordinal $O$.
For details concerning this 
construction, see \cite[p. 28]{robinson1972}.
\end{rem}

\begin{prop}\label{foreachi}
If $G$ is a hypercentral group and
$R$ is a $G$-graded ring with the 
property that for each $i < O$
the ring $R$ is $G/Z_i(G)$-graded simple,
then $R$ is simple.
\end{prop}

\begin{proof}
Take a non-zero ideal $J$ of $R$
and a non-zero $a \in J$.
We show that $\langle a \rangle = R$.
Since $\cup_i Z_i(G) = G$ and $\supp(a)$,
the support of $a$ with respect to the $G$-gradation,
is finite, we can conclude that
there is some $i$ such that $\supp(a) \subseteq Z_i(G)$.
Then $\langle a \rangle$ is a non-zero 
$G / Z_i(G)$-graded ideal of $R$.
Since $R$ is $G/Z_i(G)$-graded simple,
we get that $\langle a \rangle = R$, which shows that $J=R$.
\end{proof}

\noindent
{\bf Proof of Theorem \ref{maintheorem}.}
Suppose that $R$ is graded simple and that $Z(R)$ is a field.
Let $Z_i = Z_i(G)$, for $i \geq 0$,
denote the ascending central series of $G$
(see Remark \ref{ascending}).
By Proposition \ref{foreachi}, we are done if we can show
that $R$ is $G/Z_i$-graded simple for each $i \geq 0$.
The base case $i=0$ holds since $R$ is $G$-graded simple.
Now we show the induction step.
Suppose that the statement is true for some $i$,
i.e. that $R$ is $G / Z_i$-graded simple.
By Corollary \ref{corgenerated}, we get that
$R$ is $\frac{ G/Z_i}{ Z(G/Z_i) }$-graded simple.
Since the center of $G / Z_i$ equals 
$Z_{i+1}/Z_i$ we get that $R$ is 
$\frac{ G/Z_i }{ Z_{i+1}/Z_i }$-graded simple,
i.e. $R$ is $G/Z_{i+1}$-graded simple, 
and the induction step is complete.
\qed

\begin{defn}
For a $G$-graded ring $R$, we put $\supp(R)=\{g \in G \mid R_g \neq \{0\} \}$.
The ring $R$ said to be {\it faithfully $G$-graded} 
(see e.g. \cite{cohen1984}) if for any 
$g,h \in \supp(R)$ and any non-zero $r \in R_g$
we have that $r R_h \neq \{ 0 \} \neq R_h r$.
Recall that the {\it finite conjugate subgroup} of $G$,
denoted by $\Delta(G)$, is the set of elements of $G$
which only have finitely many conjugates (see e.g. \cite{passman1977}).
\end{defn}

\begin{prop}\label{faithfuldeltaG}
Suppose that $R$ is faithfully $G$-graded and that $\supp(R)=G$.
Then $C(R) \subseteq R_{\Delta(G)}$, and,
in particular, $Z(R) \subseteq R_{\Delta(G)}$.
\end{prop}

\begin{proof}
Take a non-zero $r \in C(R)$ and $g \in \supp(r)$.
Take $h \in G$. From the assumptions it follows that
there is $s \in R_h$ with $s r_g \neq 0$.
Since $sr = rs$ we get that 
$hg \in \supp(sr) = \supp(rs) \subseteq \supp(r)h$.
Thus $hgh^{-1} \in \supp(r)$.
Since $\supp(r)$ is finite, we get that $g \in \Delta(G)$.
The last statement follows immediately since $Z(R) \subseteq C(R)$. 
\end{proof}

\begin{cor}\label{centerRE}
If $R$ is a faithfully $G$-graded ring with $\supp(R)=G$,
where $G$ is a torsion-free hypercentral group, then $R$ is simple
if and only if $R$ is graded simple and $Z(R) \subseteq R_e$.
\end{cor}

\begin{proof}
First we show the ''if'' statement.
Suppose that $R$ is graded simple and that $Z(R) \subseteq R_e$.
We claim that $Z(R)$ is a field.
If we assume that the claim holds, then, from Theorem \ref{maintheorem},
we get that $R$ is simple. Now we show the claim.
From Section \ref{sectionnonassociativerings}, 
we know that $Z(R)$ is a commutative unital ring.
Take a non-zero $r \in Z(R)$. Since $r \in R_e$,
we get that $\langle r \rangle = R r = r R$ is a non-zero graded ideal.
By graded simplicity of $R$, we get that there is $s \in R$
with $rs = sr = 1$, which by Proposition~\ref{centerInvClosed} yields $s\in Z(R)$.
This shows that $Z(R)$ is a field.

Now we show the ''only if'' statement.
Suppose that $R$ is simple. From Theorem \ref{maintheorem},
we get that $R$ is graded simple and that $Z(R)$ is a field.
We wish to show that $Z(R) \subseteq R_e$.
From \cite[Theorem 1]{mclain1956} it follows that
$\Delta(G) = Z(G)$. By Proposition \ref{faithfuldeltaG},
we get that $Z(R) \subseteq R_{Z(G)}$.
Thus, $Z(R)$ is a $Z(G)$-graded ring.
In fact, suppose that $r \in Z(R)$ and $s \in R$.
Take $g \in G$ and $h \in Z(G)$. From the equality $r s_g = s_g r$
and the fact that $\supp(r) \subseteq Z(G)$, we get that
$r_h s_g = s_g r_h$. Summing over $g \in G$, we get that
$r_h s = s r_h$. Thus, $r_h \in Z(R)$.
It is well-known that since the group $Z(G)$ 
is torsion-free and abelian, it can be
equipped with a total order
$<$ (see e.g. \cite{Levi1942}).
Seeking a contradiction, suppose that there is a non-zero $r \in Z(R) \cap R_g$
for some $g \in Z(G) \setminus \{e\}$.
Since $Z(R)$ is a field, the element $1 + r$ must 
have a multiplicative inverse $s \in Z(R)$.
Suppose that $s = \sum_{i=1}^n s_{g_i}$,
for some $g_1,\ldots,g_n \in Z(G)$ with
$g_1 < \cdots < g_n$ and every $s_{g_i}$ non-zero
and $n$ chosen as small as possible. Then
$1 = (1 + r)s = \sum_{i=1}^n ( s_{g_i} + r s_{g_i} )$.
Case 1: $g > 0$.
Then the element $s_{g_1}$, in the last sum, has the unique smallest degree.
Thus, $1 = s_{g_1}$ and hence
$0 = r + \sum_{i=2}^n ( s_{g_i} + r s_{g_i} ).$
In this sum the element $r s_{g_n}$ has the unique largest degree.
Thus $r s_{g_n} = 0$. Since $r \neq 0$ and $Z(R)$ is a field,
we get that $s_{g_n} = 0$.
This contradicts the minimality of $n$.
Thus $Z(R) \subseteq R_e$.
Case 2: $g < 0$. This case is treated analogously to Case 1.
\end{proof}

\section{Application: Non-associative crossed products}\label{nonassociativecp}

In this section, we begin by recalling the (folkloristic) 
definitions of non-associative crossed
systems and non-associative crossed products
(see Definition \ref{defcrossedsystem} and Definition~\ref{defcrossedproduct}).
Then we show that the family of non-associative crossed products 
appear in the class of non-associative strongly graded rings in a fashion similar 
to how the associative crossed products present themselves 
in the family of associative strongly graded rings
(see Proposition~\ref{prop:crossedprod} and Proposition~\ref{prop:characterize}).
Thereafter, we determine the center of non-associative
crossed products (see Proposition \ref{centercrossed}).
At the end of this section, we use Theorem \ref{maintheorem} to obtain
non-associative versions of Theorem~\ref{bell} and Theorem~\ref{voskoglou}
(see Theorem~\ref{nonassociativebell} and
Theorem~\ref{nonassociativevoskoglou}).

\begin{defn}\label{defcrossedsystem}
A {\it non-associative crossed system} is a quadruple
$(T , G , \sigma , \alpha)$ consisting of a 
non-associative unital ring $T$, a group $G$ and
maps $\sigma : G \rightarrow {\rm Aut}(T)$ and
$\alpha : G \times G \rightarrow N(T)^{\times}$
satisfying the following three conditions for
any triple $g,h,s \in G$ and any $a \in T$:
\begin{itemize}

\item[(N1)] $\sigma_g ( \sigma_h (a) ) = \alpha(g,h) \sigma_{gh}(a) \alpha(g,h)^{-1}$;

\item[(N2)] $\alpha(g,h) \alpha(gh,s) = \sigma_g( \alpha(h,s) ) \alpha(g,hs)$;

\item[(N3)] $\sigma_e = \identity_T$ and $\alpha(g,e) = \alpha(e,g) = 1_T$.

\end{itemize}
\end{defn}

\begin{defn}\label{defcrossedproduct}
Suppose that $(T , G , \sigma , \alpha)$ is a 
non-associative crossed system.
The corresponding {\it non-associative crossed product},
denoted by $T \rtimes_{\sigma}^{\alpha} G$,
is defined as the set of finite sums $\sum_{g \in G} t_g u_g$
equipped with coordinate-wise addition and multiplication
defined by the bi-additive extension of the relations
	$(a u_g) (b u_h) = a \sigma_g(b) \alpha(g,h) u_{gh}$,
for $a,b \in T$ and $g,h \in G$.
\end{defn}

\begin{rem}
(a)
The so-called \emph{canonical $G$-gradation on $R=T \rtimes_{\sigma}^{\alpha} G$}
is obtained by putting $R_g=Tu_g$, for $g\in G$.

(b)
Notice that $R=T \rtimes_{\sigma}^{\alpha} G$
is a unital ring with identity element $u_e = 1_T u_e$.
\end{rem}

All non-associative crossed products share 
the following properties.

\begin{prop}\label{prop:crossedprod}
Let $R=T \rtimes_{\sigma}^{\alpha} G$ be a non-associative crossed product,
equipped with its canonical $G$-gradation.
The following two assertions hold:
\begin{enumerate}[{\rm (a)}]
	\item\label{propintersectionunit}
	For each $g \in G$, the intersection
$R_g \cap N(R)^{\times}$ is non-empty;
	\item\label{crossedfaithful}
	$R$ is strongly $G$-graded, and,
in particular, faithfully $G$-graded.
\end{enumerate}
\end{prop}

\begin{proof}
(a)
Take $g \in G$. 
From Proposition \ref{propNRunit} it follows that it is enough
to show that $u_g \in R_g \cap R^{\times} \cap N(R)$.
Put $a = \alpha(g,g^{-1})^{-1}$ and $v = \sigma_{g^{-1}}(a) u_{g^{-1}}$. Then 
\begin{displaymath}
	u_g v = \sigma_g ( \sigma_{g^{-1}} (a) ) \alpha(g,g^{-1}) u_e =
aa^{-1} u_e = u_e
\end{displaymath}
and from (N2) we get $\sigma_{g^{-1}}(\alpha(g,g^{-1})) = \alpha(g^{-1},g)$
which yields
\begin{displaymath}
	v u_g = \sigma_{g^{-1}}(a) \alpha(g^{-1},g) u_e = u_e.
\end{displaymath}
Therefore, $u_g \in R_g \cap R^{\times}$.
Now we show that $u_g \in N(R)$.
To this end, take $a,b \in T$ and $s,t \in G$.
First we show that $u_g \in N_l(R)$. Using (N1) and (N2), we get that
\begin{align*}
	(u_g \cdot a u_s) \cdot b u_t &=
\sigma_g(a) \alpha(g,s) u_{gs} \cdot b u_t =
\sigma_g(a) \alpha(g,s) \sigma_{gs}(b) \alpha(gs,t) u_{gst} \\
&= \sigma_g(a) \alpha(g,s) \sigma_{gs}(b) \alpha(g,s)^{-1} \alpha(g,s) \alpha(gs,t) u_{gst} \\
&\stackrel{\mathclap{\mathrm{(N1)}}}{=} \sigma_g(a) \sigma_g( \sigma_s (b) )
\alpha(g,s)
\alpha(g,st) u_{gst} 
\stackrel{\mathrm{(N2)}}{=}
\sigma_g(a) \sigma_g( \sigma_s (b) )
\sigma_g( \alpha(s,t)) \alpha(g,st) u_{gst} \\
&= \sigma_g( a \sigma_s(b) \alpha(s,t) ) \alpha(g,st) u_{gst}
= u_g \cdot a \sigma_s(b) \alpha(s,t) u_{st}
= u_g \cdot (a u_s \cdot b u_t).
\end{align*}
Now we show that $u_g \in N_m(R)$. Using (N1) and (N2), we get that
\begin{align*}
	(a u_s \cdot u_g) \cdot b u_t
	&= a \alpha(s,g) u_{sg} \cdot b u_t
= a \alpha(s,g) \sigma_{sg}(b) \alpha(sg,t) u_{sgt} \\
&= a \alpha(s,g) \sigma_{sg}(b) \alpha(s,g)^{-1} \alpha(s,g) \alpha(sg,t) u_{sgt} 
\stackrel{\mathrm{(N1)}}{=} a \sigma_s( \sigma_g (b) ) \alpha(s,g) \alpha(sg,t) u_{sgt}
\\
&\stackrel{\mathclap{\mathrm{(N2)}}}{=} a \sigma_s( \sigma_g (b) ) \sigma_s( \alpha(g,t) ) \alpha(s , gt) u_{sgt}
= a \sigma_s( \sigma_g(b) \alpha(g,t) ) \alpha(s , gt) u_{sgt} \\
&= a u_s \cdot \sigma_g(b) \alpha(g,t) u_{gt}
= a u_s \cdot ( u_g \cdot b u_t ).
\end{align*}
Finally, we show that $u_g \in N_r(R)$. Using (N2), we get that
\begin{align*}
	(a u_s \cdot b u_t) \cdot u_g
	&= a \sigma_s(b) \alpha(s,t) u_{st} \cdot u_g
= a \sigma_s(b) \alpha(s,t) \alpha(st,g) u_{stg} \\
&\stackrel{\mathclap{\mathrm{(N2)}}}{=} a \sigma_s( b ) \sigma_s( \alpha(t,g) ) \alpha(s,tg) u_{stg}
= a \sigma_s( b \alpha(t,g) ) \alpha(s , tg) u_{stg} \\
&= a u_s \cdot b \alpha(t,g) u_{tg}
= a u_s \cdot (b u_t \cdot u_g).
\end{align*}

(b)
Take $g,h \in G$ and $t \in T$.
Clearly, the inclusion $R_g R_h \subseteq R_{gh}$ holds.
Notice that
$t u_{gh} = t \alpha(g,h)^{-1} \alpha(g,h) u_{gh} = ( t \alpha(g,h)^{-1} u_g ) u_h$.
This shows that $R_{gh} \subseteq R_g R_h$.
The faithfulness of the gradation is clear.
\end{proof}

In fact, the property of
Proposition~\ref{prop:crossedprod}\eqref{propintersectionunit}
characterizes non-associative
crossed products.

\begin{prop}\label{prop:characterize}
Every $G$-graded non-associative unital ring $R$
with the property that for each $g \in G$, the intersection
$R_g \cap N(R)^{\times}$ is non-empty,
is a non-associative crossed product of the form
$T \rtimes_{\sigma}^{\alpha} G$,
associated with a non-associative crossed system $(T,G,\sigma,\alpha)$.
\end{prop}

\begin{proof}
For each $g \in G$, take $u_g \in R_g \cap N(R)^{\times}$.
Put $T=R_e$ and, supported by Proposition~\ref{NFVOPropNonAssoc}\eqref{unit}, choose $u_e=1$.

First we show that $R$, considered as a left (or right) $T$-module,
is free with $\{ u_g \}_{g \in G}$ as a basis.
From the gradation it follows that it is enough
to show that for each $g \in G$, $R_g$, 
considered as a left (or right) $T$-module
is free with $u_g$ as a basis.
Since $u_g \in R^{\times} \cap N(R)$
it follows that the left $T$-module $T u_g$
(or the right $T$-module $u_g T$) is free.
Take $g\in G$.
From Proposition~\ref{NFVOPropNonAssoc}\eqref{inverse} it follows that
$u_g^{-1} \in R_{g^{-1}}$.
Thus, $R_g = R_g u_g^{-1} u_g \subseteq T u_g \subseteq R_g$.
Hence, $R_g = T u_g$.
Analogously, 
$R_g = u_g u_g^{-1} R_g \subseteq u_g T \subseteq R_g$
which implies that $R_g = u_g T$.

Define $\sigma : G \rightarrow {\rm Aut}(T)$ by the relation
$\sigma_g(a) = u_g a u_g^{-1}$, for $g \in G$ and $a \in T$.
Define $\alpha : G \times G \rightarrow N(T)^{\times}$
by the relation $\alpha(g,h) = u_g u_h u_{gh}^{-1}$, for $g,h \in G$.
Now we check conditions (N1)--(N3) of Definition \ref{defcrossedsystem}.
To this end, take $g,h,s \in G$ and $a,b \in T$.

First we show (N1):
\begin{align*}
\sigma_g ( \sigma_h (a) ) =& u_g  u_h a u_h^{-1} u_g^{-1} =
u_g  u_h u_{gh}^{-1} u_{gh} a u_{gh}^{-1} u_{gh} u_h^{-1} u_g^{-1} \\
=& \alpha(g,h) \sigma_{gh}(a) \alpha(g,h)^{-1}
\end{align*}

Then we show (N2):
\begin{align*}
\alpha(g,h) \alpha(gh,s) =& u_g u_h u_{gh}^{-1} u_{gh} u_s u_{ghs}^{-1} =
u_g u_h u_s u_{ghs}^{-1} = 
u_g u_h u_s u_{hs}^{-1} u_{hs} u_{ghs}^{-1} \\
=& u_g \alpha(h,s) u_g^{-1} u_g  u_{hs} u_{ghs}^{-1} =
\sigma_g( \alpha(h,s) ) \alpha(g,hs)
\end{align*}

Finally, we show (N3): $\sigma_e = \identity_T$ is immediate.
Moreover, we get
\begin{displaymath}
	\alpha(g,e) = u_g u_e u_g^{-1} = 1 = u_e u_g u_g^{-1} = \alpha(e,g).
\end{displaymath}

We conclude our proof by showing that the multiplication in $R$ is compatible with
the crossed product multiplication rule:
\begin{displaymath}
	(a u_g) (b u_h) = a u_g b u_g^{-1} u_g u_h =
a \sigma_g(b) u_g u_h u_{gh}^{-1} u_{gh} = 
a \sigma_g(b) \alpha(g,h) u_{gh}.
\end{displaymath}
\end{proof}

\begin{rem}
Proposition \ref{prop:crossedprod} and Proposition \ref{prop:characterize}
are non-associative generalizations of
\cite[Proposition 1.4.1]{nastasescu2004} and \cite[Proposition 1.4.2]{nastasescu2004}, respectively.
\end{rem}

\begin{prop}\label{prop:RassocTassoc}
A non-associative crossed product $T \rtimes_{\sigma}^{\alpha} G$
is associative if and only if $T$ is associative.
\end{prop}

\begin{proof}
Put $R = T \rtimes_{\sigma}^{\alpha} G$.
Since $T$ is a subring of $R$, the ''only if''
statement is clear.
Now we show the ''if'' statement.
Suppose that $T$ is associative.
From Section \ref{sectionnonassociativerings}
we know that $N(R)$ is a subring of $R$.
Thus, from the proof of Proposition~\ref{prop:crossedprod}\eqref{propintersectionunit},
it follows that it is enough to show that $T \subseteq N(R)$.
To this end, take $a,b,t \in T$ and $g,h \in G$.

First we show that $t \in N_l(R)$. Using that $T$ is associative we get that
\begin{align*}
	t \cdot ( a u_g \cdot b u_h )
	&= t ( a \sigma_g(b) \alpha(g,h) u_{gh} )
	=  t ( a \sigma_g(b) \alpha(g,h) ) u_{gh} \\
	&=	(t a ) ( \sigma_g(b)  \alpha(g,h) ) u_{gh} =
(t a u_g) \cdot (b u_h) = 
(t \cdot a u_g)\cdot ( b u_h ).
\end{align*}
Next we show that $t \in N_m(R)$. Using that $T$ is associative we get that
\begin{align*}
	a u_g \cdot ( t \cdot b u_h )
	&= a u_g \cdot ( tb u_h )
	= a \sigma_g(tb) \alpha(g,h) u_{gh}
	= a ( \sigma_g(t) \sigma_g(b) ) \alpha(g,h) u_{gh} \\
	&=
	( a \sigma_g(t) ) \sigma_g(b) \alpha(g,h) u_{gh} =
( a \sigma_g(t) u_g ) \cdot (b u_h) = 
(a u_g \cdot t ) \cdot (b u_h).
\end{align*}
Finally, we show that $t \in N_r(R)$.
Using (N1) and the fact that $T$ is associative, we get
\begin{align*}
	(a u_g \cdot b u_h) \cdot t
	&= ( a \sigma_g(b) \alpha(g,h) u_{gh} ) \cdot t
= (a \sigma_g(b) \alpha(g,h) ) \cdot \sigma_{gh}(t) u_{gh} \\
&= a ( \sigma_g(b) \alpha(g,h) \sigma_{gh}(t) ) u_{gh}
= a ( \sigma_g(b) \alpha(g,h) (\alpha(g,h)^{-1} \sigma_g(\sigma_{h}(t)) \alpha(g,h))) u_{gh} \\
&= a ( \sigma_g(b) \sigma_g ( \sigma_h (t) ) ) \alpha(g,h) u_{gh}
= a \sigma_g( b \sigma_h(t) ) \alpha(g,h) u_{gh} \\
&= (a u_g ) \cdot (b \sigma_h(t) u_h) =
( a u_g ) \cdot ( b u_h \cdot t ).
\end{align*}
\end{proof}

\begin{defn}
Let $(T , G , \sigma , \alpha)$ be a non-associative crossed system.
An ideal $I$ of $T$ is said to be {\it $G$-invariant}
if for every $g \in G$, the inclusion $\sigma_g(I) \subseteq I$ holds.
The ring $T$ is said to be {\it $G$-simple} if 
the only $G$-invariant ideals of $T$ are $\{ 0 \}$ and $T$ itself.
\end{defn}

\begin{prop}\label{propGsimple}
A non-associative crossed product $T \rtimes_{\sigma}^{\alpha} G$
is graded simple, with respect to its canonical $G$-gradation, if and only if $T$ is $G$-simple.
\end{prop}

\begin{proof}
Put $R = T \rtimes_{\sigma}^{\alpha} G$.
First we show the ''only if'' statement.
Suppose that $R$ is graded simple.
Let $J$ be a non-zero $G$-invariant ideal of $T$.
Put $I = JR$. Then $I$ is a non-zero 
graded ideal of $R$.
From graded simplicity of $R$ it follows that $I = R$.
Therefore, $J = T$.
Now we show the ''if'' statement.
Suppose that $T$ is $G$-simple.
Let $I$ be a non-zero graded ideal of $R$.
Consider the ideal $J = I \cap T$ of $T$.
From the proof of Proposition~\ref{prop:crossedprod}\eqref{propintersectionunit}
we notice that $u_g \in N(R)^\times$, for each $g\in G$, and hence $J$ is non-zero.
Take a non-zero $t \in J$.
From the equalities 
$\sigma_g(t) = u_g t u_{g^{-1}} \alpha(g,g^{-1})^{-1}$,
for $g \in G$, it follows that $J$ is $G$-invariant.
By $G$-simplicity of $T$ we get that $J=I \cap T = T$.
In particular, $1\in I$ and hence $I=R$.
\end{proof}

\begin{prop}\label{centercrossed}
The center of a non-associative crossed product 
$T \rtimes_{\sigma}^{\alpha} G$ equals
the set of elements of the form $\sum_{g \in G} t_g u_g$,
with $t_g \in T$, such that
for all $g,h \in G$ and all $t \in T$, 
the following four properties hold:
\begin{itemize}

\item[(i)] $t t_g = t_g \sigma_g(t)$;

\item[(ii)] $t_{h g h^{-1}} = \sigma_h(t_g) \alpha(h,g) \alpha(h g h^{-1} , h)^{-1}$;

\item[(iii)] $t_g \in N(T)$;

\item[(iv)] if $g \notin \Delta(G)$, then $t_g = 0$.

\end{itemize}
\end{prop}

\begin{proof}
Put $R = T \rtimes_{\sigma}^{\alpha} G$ and
let $r = \sum_{g \in G} t_g u_g \in R$. 
Take $s,t \in T$ and $h \in G$.

Suppose first that $r \in Z(R)$.
Property (i) follows from the 
equality $tr = rt$.
Property (ii) follows from the 
equality $r u_h = u_h r$.
Property (iii) follows from the equalities 
$(st)r = s(tr)$, $(sr)t = s(rt)$ and $(rs)t = r(st)$.
Property (iv) follows from
property (ii) above.

Now suppose that $r$ satisfies
properties (i), (ii), (iii) and (iv).
We wish to show that $r \in Z(R)$.
From (i) and (ii), and the fact that $u_h \in N(R)$ (see the proof of Proposition~\ref{prop:crossedprod}\eqref{propintersectionunit}), 
we get that $r \in C(R)$.
From (iii), and the fact that $u_g \in N(R)$,
for all $g \in G$, we get that $r \in N(R)$.
\end{proof}

\begin{prop}\label{propZTG}
If $T \rtimes_{\sigma}^{\alpha} G$ is a non-associative crossed product,
where $T$ is $G$-simple, then $Z(T)^G := \{t\in Z(T) \mid \sigma_g(t)=t, \ \forall g\in G\}$ is a field.
\end{prop}

\begin{proof}
Clearly, $Z(T)^G$ is a commutative unital ring.
Now we show that every non-zero element in $Z(T)^G$
has a multiplicative inverse.
To this end, put $R = T \rtimes_{\sigma}^{\alpha} G$.
Take a non-zero $t \in Z(T)^G$.
Put $I = tT = Tt$.
Then $I$ is a non-zero $G$-invariant ideal of $T$.
By $G$-simplicity of $T$, we get that $I = Z(T)^G$.
In particular, there is $s \in T$ such that $st = ts = 1$.
By Proposition \ref{centerInvClosed}, we get that $s \in Z(T)$.
Take $g \in G$. Then
$\sigma_g(s) = 
\sigma_g(s) 1 = 
\sigma_g(s) ts = 
\sigma_g(s) \sigma_g(t) s =
\sigma_g(st) s =
\sigma_g(1) s = 1s = s$.
Hence, $s \in Z(T)^G$.
\end{proof}

\begin{lem}\label{lemmamclain}
If $G$ is a torsion-free hypercentral group,
then $\Delta(G) = Z(G)$.
\end{lem}

\begin{proof}
See \cite[Theorem 1]{mclain1956}.
\end{proof}

The following result generalizes \cite[Theorem 5.6]{bell1987}.

\begin{thm}\label{nonassociativebell}
If $T \rtimes_{\sigma}^{\alpha} G$ is a non-associative crossed product,
where $G$ is a torsion-free hypercentral group, 
then the following three conditions are equivalent:
\begin{itemize}

\item[(i)] $T \rtimes_{\sigma}^{\alpha} G$ is simple;

\item[(ii)] $T$ is $G$-simple and $Z( T \rtimes_{\sigma}^{\alpha} G ) = Z(T)^G$;

\item[(iii)] $T$ is $G$-simple and there do not exist $u \in T^{\times}$
and $g \in Z(G) \setminus \{e\}$ such that for every
$h \in G$ and every $t \in T$, the relations
$\sigma_h(u) = u \alpha(g,h) \alpha(h,g)^{-1}$ and
$\sigma_g(t) = u^{-1} t u$ hold.
\end{itemize} 
\end{thm}

\begin{proof}
Put $R = T \rtimes_{\sigma}^{\alpha} G$ with its canonical $G$-gradation, which, by Proposition~\ref{prop:crossedprod}, is faithful.

(i)$\Rightarrow$(ii):
It follows from Proposition~\ref{propGsimple}
that $T$ is $G$-simple.
From Corollary~\ref{centerRE}, we get that 
$Z( R ) \subseteq T$.
From Proposition \ref{centercrossed}, we 
thus get that $Z(R) = Z(T)^G$.

(ii)$\Rightarrow$(i):
We notice that $Z(R)\subseteq R_e=T$.
Hence, the desired conclusion follows from 
Proposition~\ref{propGsimple} and Corollary \ref{centerRE}.

(iii)$\Rightarrow$(i):
We claim that $Z(R) \subseteq R_e$.
If we assume that the claim holds, then (i) follows from
Proposition~\ref{propGsimple} and Corollary \ref{centerRE}.
Now we show the claim.
Take $r = \sum_{h \in G} t_h u_h \in Z(R)$.
Seeking a contradiction, suppose that there 
is some $g \in \supp(r) \setminus \{e\}$.
By Proposition \ref{centercrossed} and Lemma \ref{lemmamclain},
it follows that $g \in Z(G)$.
Put $u = t_g$.
From Proposition \ref{centercrossed}(i),
we get that $Tu = uT$ is a non-zero 
$G$-invariant ideal of $T$.
Thus, from $G$-simplicity of $T$,
we get that $u \in T^{\times}$.
This contradicts Proposition \ref{centercrossed}(i)--(ii).

(i)$\Rightarrow$(iii):
Seeking a contradiction, suppose that (iii) fails to hold.
Then there is some $u \in T^{\times}$ and $g \in Z(G)\setminus\{e\}$
such that for every $h \in G$ and every $t \in T$, the relations
$\sigma_h(u) = u \alpha(g,h) \alpha(h,g)^{-1}$ and
$\sigma_g(t) = u^{-1} t u$ hold.
By Proposition \ref{centercrossed}, the element
$1 + u u_g$ belongs to $Z(R) \setminus R_e$.
This contradicts Corollary \ref{centerRE}.
\end{proof}

\begin{lem}\label{lemunit}
Suppose that $T \rtimes_{\sigma}^{\alpha} G$ is a non-associative
crossed product and that $H$ is a subgroup of $G$.
Let $\sigma'$ and $\alpha'$ be the restriction of $\sigma$ and $\alpha$
to $H$ and $H\times H$, respectively.
Then $T \rtimes_{\sigma'}^{\alpha'} H$ is a non-associative
crossed product and the map
$\pi : T \rtimes_{\sigma}^{\alpha} G \rightarrow
T \rtimes_{\sigma'}^{\alpha'} H$,
defined by $\pi( \sum_{g \in G} t_g u_g ) = \sum_{h \in H} t_h u_h$, 
is a $T \rtimes_{\sigma'}^{\alpha'} H$-bimodule homomorphism.
Hence, if $r \in (T \rtimes_{\sigma'}^{\alpha'} H) \cap (T \rtimes_{\sigma}^{\alpha} G)^\times$,
then $r \in (T \rtimes_{\sigma'}^{\alpha'} H)^\times$. 
\end{lem}

\begin{proof}
The argument used for the proofs of Lemma 1.2 and Lemma 1.4 in \cite{passman1977},
for associative group rings, can easily be adapted to the case of 
non-associative crossed products.
\end{proof}

\begin{defn}
If a non-associative crossed product $T \rtimes_{\sigma}^{\alpha} G$
satisfies $\alpha(g,h)=1$, for all $g,h \in G$, then
we call it a {\it non-associative skew group ring} 
and denote it by $T \rtimes_{\sigma} G$;
in this case $\sigma : G \rightarrow {\rm Aut}(T)$
is a group homomorphism.
On the other hand, if a non-associative crossed product $T \rtimes_{\sigma}^{\alpha} G$
satisfies $\sigma_g = \identity_T$, for all $g \in G$, then
we call it a {\it non-associative twisted group ring} 
and denote it by $T \rtimes^{\alpha} G$.
\end{defn}

In the sequel, we shall refer to the following set
\begin{displaymath}
	L = \{ g \in G \mid \mbox{ $\sigma_g$ acts as conjugation
by an invertible element of $N(T)^G$ } \}.
\end{displaymath}

\begin{lem}\label{lemmacentertorsion}
Suppose that $T \rtimes_{\sigma} G$ is a non-associative skew group ring,
where $G$ is an abelian group and $T$ is $G$-simple.
Then $Z(T \rtimes_{\sigma} G)$ equals a 
twisted group ring $F \rtimes^{\alpha} L$,
where $F = T^G \cap Z(T)$.
\end{lem}

\begin{proof}
Put $R = T \rtimes_{\sigma} G$.
Since $G$ is abelian, it follows from Proposition \ref{centercrossed} that
$Z(R) = \oplus_{g\in G} T_g u_g$, where for each $g \in G$, $T_g$
is the set of $t_g \in T^G \cap N(T)$ satisfying $t_g \sigma_g(t) = t t_g$, for $t \in T$.
From $G$-simplicity of $T$ it follows that each non-zero
$t_g \in T_g$ is invertible.
Thus, $Z(R) = \oplus_{g \in L} T_g u_g$.
Put $F = T_e$.
If $g \in L$ and $t_g,s_g \in T_g$ are non-zero,
then, for each $t \in T$, we have that 
$\sigma_g^{-1} \sigma_g (t) = t$.
Thus, $s_g t_g^{-1} t t_g s_g^{-1} = t$ and hence
$t t_g s_g^{-1} = t_g s_g^{-1} t$ from which it follows that
$t_g s_g^{-1} \in F$.
Therefore $t_g \in F s_g$.
We can thus write $Z(R) = \oplus_{g \in L} F d_g$,
where $d_g = s_g u_g$.
Take $g,h \in L$. Then it is clear that
$d_g d_h = \alpha(g,h) d_{gh}$, where
$\alpha(g,h) = s_g s_h s_{gh}^{-1}$.
Now we show that $\alpha(g,h) \in F$.
Since $s_g , s_h , s_{gh} \in T^G$,
it follows that $\alpha(g,h) \in T^G$.
Next we show that $\alpha(g,h) \in C(T)$.
Take $t \in T$. Then 
$t \alpha(g,h) = 
t s_g s_h s_{gh}^{-1} =
s_g \sigma_g(t) s_h s_{gh} ^{-1} =
s_g s_h \sigma_{hg}(t) s_{gh}^{-1} =
s_g s_h \sigma_{gh}(t) s_{gh}^{-1} =
s_g s_h s_{gh}^{-1} t = \alpha(g,h) t$.
Since we already know that $s_g , s_h , s_{gh}^{-1} \in N(T)$,
we thus get that $\alpha(g,h) \in Z(T)$.
Finally, we show (N2) of Definition \ref{defcrossedsystem}.
To this end, take $g,h,p \in G$. Then 
\begin{align*}
\alpha(g,h) \alpha(gh , p) =& 
s_g s_h s_{gh}^{-1} s_{gh} s_p   s_{ghp}^{-1} =
s_g s_h                    s_p   s_{ghp}^{-1} =
s_g \sigma_g(s_h s_p) s_{ghp}^{-1} =
s_h s_p s_g s_{ghp}^{-1} \\
=& s_h s_p \sigma_{hp}(s_g)         s_{ghp}^{-1} =
s_h s_p s_{hp}^{-1} s_g s_{hp}   s_{ghp}^{-1} =
\alpha(h,p) \alpha(g,hp).
\end{align*}
Thus $Z(R) = F \rtimes^{\alpha} L$.
\end{proof}

The next result is a non-associative generalization
of \cite[Proposition 5.7]{bell1987}.

\begin{prop}
Suppose that $T \rtimes_{\sigma} G$ is a non-associative skew group ring,
where $G$ is an abelian group.
The following two assertions hold:
\begin{itemize}

\item[(a)] $T \rtimes_{\sigma} G$ is simple if and only if
$T$ is $G$-simple, 
$Z( T \rtimes_{\sigma}^{\alpha} G )$ is a domain, and
$L$ is a torsion group;

\item[(b)] If $T$ is commutative, then 
$T \rtimes_{\sigma} G$ is simple if and only if
$T$ is $G$-simple and $L = \{ e \}$.

\end{itemize}
\end{prop}

\begin{proof}
Put $R = T \rtimes_{\sigma} G$.

(a) Suppose that $R$ is simple.
From Theorem \ref{maintheorem} and Proposition \ref{propGsimple}, 
it follows that $T$ is $G$-simple and that $Z(R)$ is a field
(hence a domain). Now we show that $L$ is a torsion group.
Seeking a contradiction, suppose that there is $g \in L$
of infinite order and that there is an invertible $u \in T^G$
such that $\sigma_g$ acts as conjugation by $u$.
By Proposition \ref{centercrossed}, the element
$x = 1 + u u_g$ belongs to $Z(R)$.
Put $S = T \rtimes_{\sigma'}^{\alpha'} \langle g \rangle$.
By Lemma \ref{lemunit}, we get that $x$ is invertible in $S$.
Suppose that $x^{-1} = \sum_{i=a}^b t_i u_{g^i}$,
where $a,b \in \mathbb{Z}$, satisfy $a \leq b$ and
both $t_a$ and $t_b$ are non-zero.
From the equality $x^{-1} x = 1$, we get, since $g$
is of infinite order, that either
$t_a u_{g^a} = 0$ or $t_b \sigma_{g^b}(u) \alpha(g^b,g) u_{g^{b+1}} = 0$.
This is clearly not possible.
Thus, $L$ is a torsion group.

Now suppose that $T$ is $G$-simple, 
$Z( R )$ is a domain, and that $L$ is a torsion group.
From Theorem \ref{maintheorem} and Proposition~\ref{propGsimple}, it follows that
$R$ is simple if we can show that $Z(R)$ is a field.
To this end, by Proposition~\ref{centerInvClosed}, it is enough to show that 
each non-zero element of $Z(R)$ has a multiplicative inverse in $R$.
From Lemma \ref{lemmacentertorsion} it follows that
$Z(R) = F \rtimes^{\alpha} L$.
Since $L$ is a torsion group it follows that
every $x = f d_g$, $f \in F$, $g \in L$, satisfies
$x^n \in F$ for some positive integer $n$.
Since integral elements over an integral domain
form a ring, it follows that any non-zero 
element $x \in Z(R)$ is integral over $F$.
Let $K$ denote the field of fractions of $Z(R)$.
Then $x^{-1}$, considered as an element of $K$,
belongs to $F[x] \subseteq Z(R)$.

(b) Suppose that $R$ is simple.
From (a) we get that $T$ is $G$-simple
and that $L$ is a torsion group.
From the proof of (a), we get that
the field $Z(R)$ equals the group ring $F[L]$.
Seeking a contradiction, suppose that there is
some $g \in L \setminus \{e\}$.
Since $L$ is torsion, there is an integer $n \geq 2$
such that $g^n = 1$.
But then $(1-g)(1 + g + \cdots + g^{n-1}) = 1 - g^n = 0$.
Therefore, $1-g$ can not be invertible.
Hence $L = \{ e \}$.

Now suppose that $T$ is $G$-simple and that $L = \{ e \}$.
From the proof of (a) we get that $Z(R) = F$
which is a field and, hence, a domain.
Thus, simplicity of $R$ follows from (a).
\end{proof}

\begin{defn}
Suppose that $T \rtimes_{\sigma} G$ is a non-associative
skew group ring.
If $G$ is a torsion-free finitely generated
abelian group, with generators $g_1,\ldots,g_n$, for some
positive integer $n$, then  
$T \rtimes_{\sigma} G$ equals the 
{\it non-associative skew Laurent polynomial ring}
$T[x_1,\ldots,x_n,x_1^{-1},\ldots,x_n^{-1} ; \sigma_1,\ldots,\sigma_n],$
where $\sigma_i := \sigma_{g_i}$, for $i \in \{1 , \ldots , n\}$.
In that case, $T$ is called \emph{$(\sigma_1,\ldots,\sigma_n)$-simple}
if there is no ideal $I$ of $T$, other than $\{ 0 \}$ and $T$,
for which $\sigma_i(I) \subseteq I$ holds for all $i \in \{1,\ldots,n\}$.
\end{defn}

The following result generalizes
\cite[Theorem 1]{jordan1984} and the main result in \cite{voskoglou1987}.

\begin{thm}\label{nonassociativevoskoglou}
A non-associative skew Laurent polynomial ring
\begin{displaymath}
	T[x_1,\ldots,x_n,x_1^{-1},\ldots,x_n^{-1} ; \sigma_1,\ldots,\sigma_n]
\end{displaymath}
is simple if and only if $T$ is $(\sigma_1,\ldots,\sigma_n)$-simple
and there do not exist $u \in \cap_{i=1}^n ( T^{\times} )^{\sigma_i}$
and a non-zero $(m_1,\ldots,m_n) \in \mathbb{Z}^n$
such that for every $t \in T$,
the relation $( \sigma_1^{m_1} \circ \cdots \circ \sigma_n^{m_n} )(t) = 
u t u^{-1}$ holds. 
\end{thm}

\begin{proof}
We notice that $\Z^n$-simplicity of $T$
is equivalent to $(\sigma_1,\ldots,\sigma_n)$-simplicity of $T$.
The proof now
follows from Theorem \ref{nonassociativebell}.
\end{proof}

\section{Application: Cayley-Dickson doublings}\label{cayleydickson}

In this section, we apply Theorem \ref{maintheorem} to
obtain a simplicity result for Cayley-Dickson doublings
(see Theorem \ref{genmccrimmon}) originally 
obtained by K. McCrimmon \cite[Theorem 6.8(vi)]{mccrimmon1985} by other means.
To do this we first need to translate what graded simplicity
means for Cayley-Dickson doublings (see Proposition \ref{translategradedsimple}).
After that, we describe the center of Cayley-Dickson doublings
(see Proposition \ref{translatecenter}).

Throughout this section, $A$ denotes
a non-associative unital algebra with involution
$* : A \rightarrow A$ 
defined over an arbitrary unital, commutative and associative ring of scalars $K$.
Take a self-adjoint cancellable scalar $\mu \in K$.
By this we mean that $\mu^*=\mu$ and that
whenever $\mu k = 0$ (or $\mu a = 0$), 
for some $k \in K$ (or some $a \in A$),
then $k = 0$ (or $a = 0$).
We can now construct a new algebra,
the so-called \emph{Cayley Double of $A$}, denoted by $C(A,\mu)$, as
$A \oplus A$ with involution 
$(a , b)^* = (a^* , -b)$ 
and product 
$(a , b)(c , d) = (ac + \mu d^* b , da + b c^*).$
We can write this formally as
$C(A,\mu) = A \oplus A l$ with product
$(a + bl)(c + dl) = (ac + \mu d^* b) + (da + b c^*)l$
and involution $(a+bl)^* = a^* - bl.$
An ideal $I$ of $A$ is called \emph{$*$-invariant}
if $I^* \subseteq I$. The ring $A$ is called
\emph{$*$-simple} if there are no $*$-invariant
ideals of $A$ other that $\{ 0 \}$ and $A$ itself.
It is clear that if we put $R_0 = A$ and
$R_1 = Al$, then this defines a $\mathbb{Z}/2\mathbb{Z}$-gradation on $R = C(A,\mu)$.

\begin{prop}\label{translategradedsimple}
The ring $C(A,\mu)$ is graded simple 
if and only if $A$ is $*$-simple and $\mu \in Z(A)^{\times}$.
\end{prop}

\begin{proof}
We first show the ''only if'' statement.
Suppose that $C(A,\mu)$ is graded simple.
Let $I$ be a non-zero $*$-invariant ideal of $A$.
Then $\langle I \rangle = I + Il$ is a non-zero graded ideal of $C(A,\mu)$.
By graded simplicity of $C(A,\mu)$ we get that 
$\langle I \rangle = C(A,\mu)$. In particular, we get that $1 \in \langle I \rangle$.
This implies that $1 \in \langle I \rangle \cap A = I$.
Hence, $I = A$. This shows that $A$ is $*$-simple.

Using that $\mu \in Z(A)$ and $\mu = \mu^*$ we get that
$A \mu$ is a non-zero $*$-invariant ideal of $A$.
Hence, $A \mu = A$, and in particular there is some $s\in A$ such that
$s\mu=\mu s = 1$. Using Proposition~\ref{centerInvClosed} we conclude that $\mu \in Z(A)^\times$.

Now we show the ''if'' statement. Suppose that $A$ is $*$-simple and that $\mu \in Z(A)^\times$. 
Let $J$ be a non-zero graded ideal of $C(A,\mu)$.
Put $I = J \cap A$.
Using that $\mu$ is cancellable, we get that $I$ is a non-zero ideal of $A$.
The multiplication rule in $C(A,\mu)$ yields $ldl=\mu d^*$, for any $d\in A$.
Using that $\mu\in Z(A)^\times$ we get
$I^* = \mu^{-1} l I l \subseteq I$,
and hence $I$ is $*$-invariant.
By the $*$-simplicity of $A$ we get that 
$I = A$. Therefore $A \subseteq J$ and hence $J = C(A,\mu)$.
This shows that $C(A,\mu)$ is graded simple.
\end{proof}

\begin{prop}\label{mccrimmonprop}
The center of $C(A,\mu)$ equals
$Z_*(A) + Z_{**}(A)l$, where 
$Z_*(A) = \{ a \in Z(A) \mid a = a^* \}$ and
$Z_{**}(A) = \{ a \in Z_*(A) \mid ab=ab^*,\ \forall b\in A \}$.
\end{prop}

\begin{proof}
This is \cite[Theorem 6.8(xii)]{mccrimmon1985}.
\end{proof}

\begin{prop}\label{translatecenter}
The ring $Z(C(A,\mu))$ is a field if and only if either
(i) $* = \identity_A$, $Z(A)$ is a field
and $\mu$ is not a square in $Z(A)$, or
(ii) $* \neq \identity_A$ and $Z_*(A)$ is a field.
\end{prop}

\begin{proof}
From Proposition \ref{mccrimmonprop}, we get that
$Z(C(A,\mu)) = Z_*(A) + Z_{**}(A)l$.

First we show the ''only if'' statement.
Suppose that $Z(C(A,\mu))$ is a field.
It is clear that $Z_*(A)$ is a field.
Since $Z_{**}(A)$ is an ideal of $Z_*(A)$ we get two cases.
Case 1: $Z_{**}(A) = Z_*(A)$.
Since $1 \in Z_*(A)$, we get that $*=\identity_A$
and hence $Z(A) = Z_*(A) = Z_{**}(A)$ is a field.
Therefore, $Z(C(A,\mu)) = Z(A) + Z(A) l = Z(A)[X]/(X^2-\mu)$.
If $\mu$ is a square in $Z(A)$, then 
$Z(C(A,\mu)) = Z(A) \times Z(A)$ or $l^2 = 0$, 
depending on the characteristic of $Z(A)$. 
In either case,
$Z(C(A,\mu))$ is not a field.
Hence we get that $\mu$ is not a square in $Z(A)$.
Thus, (i) holds.
Case 2: $Z_{**}(A) = \{ 0 \}$. 
In this case $Z(C(A,\mu)) = Z_*(A)$ is a field.
Seeking a contradiction, suppose that $* = \identity_A$.
Then $ab=ab^*$, for all $a,b\in A$, and hence $Z_*(A)=Z_{**}(A) = \{0\}$
which is a contradiction. Thus, $*\neq \identity_A$.
This shows that (ii) holds.

Now we show the ''if'' statement.
First, suppose that (i) holds.
Then $Z_{**}(A) = Z_*(A) = Z(A)$ and thus
$Z(C(A,\mu)) = Z(A) + Z(A)l = Z(A)[X]/(X^2-\mu)$,
which is a field, since $\mu$ is not a square in $Z(A)$.
Now suppose that (ii) holds.
Since $Z_{**}(A)$ is an ideal of the field $Z_*(A)$
we get that either $Z_{**}(A) = Z_*(A)$ or $Z_{**}(A) = \{ 0 \}$.
If $Z_{**}(A) = Z_*(A)$, then $*=\identity_A$ which is a contradiction.
Therefore, $Z_{**}(A) = \{ 0 \}$ and hence
$Z(C(A,\mu)) = Z_*(A)$ is a field.
\end{proof}

\begin{thm}\label{genmccrimmon}
The ring $C(A,\mu)$ is simple if and only if $A$ is $*$-simple and either
(i) $A$ has trivial involution, $Z(A)$ is a field
and $\mu$ is not a square in $Z(A)$, or
(ii) $A$ has non-trivial involution and $Z_*(A)$ is a field.
\end{thm}

\begin{proof}
The ''only if'' statement follows immediately from Proposition~\ref{translategradedsimple}, Proposition~\ref{translatecenter} and Theorem~\ref{maintheorem}.
The ''if'' statement follows in the same way, by first observing that
$\mu \in Z_*(A) \subseteq Z(A)$, and that $\mu$ therefore (in either case) is a non-zero element of a field, and thus invertible.
\end{proof}

\end{document}